\documentclass[10pt]{article}
 \usepackage{amssymb}
 \usepackage{amsmath}
 \usepackage{srcltx}
\usepackage{amsfonts}
\usepackage{euscript}
\usepackage{epsfig}
\usepackage{amsthm}

\newtheorem*{thm}{Theorem}
\newtheorem*{cor1}{Corollary 1}
\newtheorem*{cor2}{Corollary 2}
\newtheorem*{cor3}{Corollary 3}
\newtheorem*{cor4}{Corollary 4}
\newtheorem*{cor5}{Corollary 5}
\newtheorem*{sublemma}{Sublemma}

\theoremstyle{definition}

\newtheorem*{eg}{Example}

\theoremstyle{remark}

\numberwithin{equation}{section}
\begin{document}

\title{\large{\textbf{On a class of hereditary crossed-product orders\footnote{\textbf{Accepted for publication by the Proceedings of the American Mathematical Society.}}} } }

\author{\bf \sf John S. Kauta}

\date{}

 \maketitle

\begin{abstract}
In this brief note, we revisit a class of crossed-product orders
over discrete valuation rings introduced by D. E. Haile. We give simple but
useful criteria, which involve only the two-cocycle associated with
a given crossed-product order, for determining whether such an order
is a hereditary order or a maximal order.

\end{abstract}

If $R$ is a ring, then $J(R)$ will denote its Jacobson radical,
$U(R)$ its group of multiplicative units, and $R^{\#}$ the subset of
all the non-zero elements. The terminology used in this paper, if
not in \cite{H}, can be found in \cite{R}. The book by Reiner
\cite{R} is also an excellent source of literature on maximal orders
and hereditary orders.

Let $V$ be a discrete valuation ring (DVR), with quotient field $F$,
and let $K/F$ be a finite Galois extension, with group $G$, and let
$S$ be the integral closure of $V$ in $K$. Let $f\in Z^2(G,U(K))$ be
a normalized two-cocycle. If $f(G\times G)\subseteq S^{\#}$, then
one can construct a ``crossed-product'' $V$-algebra
$$A_f=\sum_{\sigma\in G} Sx_{\sigma},$$ with the usual rules of multiplication
($x_{\sigma}s=\sigma(s)x_{\sigma}$ for all $s\in S,\sigma\in G$ and
$x_{\sigma}x_{\tau}=f(\sigma,\tau)x_{\sigma\tau})$. Then $A_f$ is
associative, with identity $1=x_1$, and center $V=Vx_1$. Further,
$A_f$ is a $V$-order in the crossed-product $F$-algebra
$\Sigma_f=\sum_{\sigma\in G} Kx_{\sigma}=(K/F,G,f)$.

Two such cocycles $f$ and $g$ are said to be cohomologous over $S$
(respectively cohomologous over $K$), denoted by $f\! \sim_S\! g$
(respectively $f\! \sim_K\! g$), if there are elements
$\{c_{\sigma}\mid \sigma\in G\}\subseteq U(S)$ (respectively
$\{c_{\sigma}\mid \sigma\in G\}\subseteq K^{\#}$) such that
$g(\sigma,\tau)=c_{\sigma}\sigma{(c_{\tau})}c_{\sigma\tau}^{-1}f(\sigma,\tau)$
for all $\sigma,\tau \in G$. Following \cite{H}, let $H=\{\sigma\in
G\mid f(\sigma,\sigma^{-1})\in U(S)\}$. Then $H$ is a subgroup of
$G$. On $G/H$, the left coset space of $G$ by $H$, one can define a
partial ordering by the rule $\sigma H \leq \tau H
\;\;\textrm{if}\;\; f(\sigma,\sigma^{-1}\tau)\in U(S).$ Then ``$\leq
$'' is well-defined and depends only on the cohomology class of $f$
over $S$. Further, $H$ is the unique least element. We call this
partial ordering on $G/H$ \textit{the graph of f.}

Such a setup was first formulated by Haile in \cite{H}, with the
assumption that $S$ is unramified over $V$, wherein, among other
things, conditions equivalent to such orders being maximal orders
were considered. This is the class of crossed-product orders we
shall study in this paper, \textit{always assuming that $S$ is
unramified over $V$}. We emphasize the fact that, since we do not
require that $f(G\times G)\subseteq U(S)$, this theory constitutes a
drastic departure from the classical theory of crossed-product
orders over DVRs, such as can be found in \cite{Hd}.

Let us now fix additional notation to be used in the rest of the
paper, most of it borrowed from \cite{H} as before. If $M$ is a
maximal ideal of $S$, let $D_M$ be the decomposition group of $M$,
let $K_M$ be the decomposition field, and let $S_M$ be the
localization of $S$ at $M$. The two-cocycle $f:G\times G\mapsto
S^{\#}$ yields a two-cocycle $f_M:D_M\times D_M\mapsto S^{\#}_M$,
determined by the restriction of $f$ to $D_M\times D_M$ and the
inclusion of $S^{\#}$ in $S_M^{\#}$. Then $A_{f_M}=\sum_{\sigma\in
D_M}S_Mx_{\sigma}$ is a crossed-product order in
$\Sigma_{f_M}=\sum_{\sigma\in D_M}Kx_{\sigma}=(K/K_M,D_M,f_M)$. In
addition, we can obtain a \textit{twist} of $f$, described in
\cite[pp. 137-138]{H} and denoted by $\tilde{f}$, which depends on
the choice of a maximal ideal $M$ of $S$, and the choice of a set of
coset representatives of $D_M$ in $G$. We also define $F:G\times
G\mapsto S^{\#}$ by $F(\sigma,\tau)=f(\sigma,\sigma^{-1}\tau)$ for
$\sigma,\tau\in G$. While $\tilde{f}$ is a two-cocycle, $F$ is not.

If $B$ is a $V$-order of $\Sigma_f$ containing $A_f$, then by
\cite[Proposition 1.3]{H}, $B=A_g=\sum_{\sigma\in G}Sy_{\sigma}$ for
some two-cocycle $g:G\times G\mapsto S^{\#}$, with $g\!\sim_K\!f$.
Moreover, the proof of \cite[Proposition 1.3]{H} shows that
$y_{\sigma}=k_{\sigma}x_{\sigma}$ for some $k_{\sigma}\in K^{\#}$,
with $k_1=1$, whence $g$ is also a normalized two-cocycle.

We begin with a technical result.

 \begin{sublemma}
Let $\tau\in G$. If
$\displaystyle{I_{\tau}=\hspace*{-1.5em}\prod_{f(\tau,\tau^{-1})\not\in
M}\hspace*{-1.5em}M}$, where $M$ denotes a maximal ideal of $S$,
then $I_{\tau}^{\tau^{-1}}=I_{\tau^{-1}}$.
\end{sublemma}

\begin{proof} We have
$$I_{\tau}^{\tau^{-1}}=\hspace*{-1.5em}\prod_{f(\tau,\tau^{-1})\not\in
M}\hspace*{-1.6em}M^{\tau^{-1}}=\hspace*{-1.5em}\prod_{f^{\tau^{-1}}(\tau,\tau^{-1})\not\in
M^{\tau^{-1}}}\hspace*{-3em}M^{\tau^{-1}}=\hspace*{-1em}\prod_{\hspace*{1em}f(\tau^{-1},\tau)\not\in
M^{\tau^{-1}}}\hspace*{-2.8em}M^{\tau^{-1}}=I_{\tau^{-1}}.$$
\end{proof}

\begin{thm} The crossed-product order $A_f$ is hereditary if and only if
$f(\tau,\tau^{-1})\not\in M^2$ for all $\tau\in G$ and every
maximal ideal $M$ of $S$.\end{thm}

\begin{proof} The theorem obviously holds if $H=G$, in which case
$A_f$ is an Azumaya algebra over $V$, so let us assume from now on
that $H\not= G$.

Suppose $A_f$ is hereditary. First,  assume $A_f$ is a maximal
order and $S$ is a DVR. Let $v$ be the valuation corresponding to $S$ with
value group $\mathbb{Z}$. Then
by \cite[Theorem~2.3]{H}, $H$ is a normal subgroup of $G$ and $G/H$ is
cyclic. Further, there exists $\sigma\in G$ such that
$v(f(\sigma,\sigma^{-1}))\leq 1$, $G/H=<\sigma H>$, and the graph of
$f$ is the chain $H\leq\sigma H\leq\sigma^2
H\leq\cdots\leq\sigma^{m-1} H$, where $m=\mid G/H\mid$. Choose $j$
maximal such that $1\leq j\leq m-1$ and
$v(f(\sigma^i,\sigma^{-i}))\leq 1\;\forall\; 1\leq i\leq j$. We
always have $\sigma H\leq \sigma^{-j} H$; but if $j<m-1$, then we
also have $\sigma^j H\leq \sigma^{j+1}H$. Hence if $j < m-1$, then, from the cocycle
identity
$f^{\sigma^j}(\sigma,\sigma^{-j}\sigma^{-1})f(\sigma^j,\sigma^{-j})=
f(\sigma^j,\sigma)f(\sigma^{j+1},\sigma^{-j}\sigma^{-1}),$ we
conclude that $v(f(\sigma^i,\sigma^{-i}))\leq 1\;\forall\; 1\leq
i\leq j+1$, a contradiction. So we must have $j=m-1$, so that
$v(f(\sigma^i,\sigma^{-i}))\leq 1\;\forall\; 1\leq i\leq m-1$. If
$\tau$ is an arbitrary element of $G$, then $\tau=\sigma^ih$ for
some $h\in H$ and some integer $i$, $0\leq i\leq m-1$. Therefore, by
\cite[Lemma 3.6]{H},
$v(f(\tau,\tau^{-1}))=v(F(\sigma^ih,1))=v(F(\sigma^i,1))=v(f(\sigma^i,\sigma^{-i}))\leq
1$; that is, $f(\tau,\tau^{-1})\not\in J(S)^2$.

We maintain the assumption that $A_f$ is a maximal order, but we now
drop the condition that $S$ is a DVR. By \cite[Theorem 3.16]{H},
there exists a twist of $f$, say $\tilde{f}$, such that
$f\!\sim_S\!\tilde{f}$. By \cite[Corollary 3.11]{H}, for any maximal
ideal $M$ of $S$, $A_{f_M}$ is a maximal order in $\Sigma_{f_M}$;
hence $f_M(\tau,\tau^{-1})\not\in M^2\;\forall\;\tau\in D_M$ by the
preceding paragraph. Therefore, from the manner in which $\tilde{f}$
is constructed from $f$, we infer that
$\tilde{f}(\tau,\tau^{-1})\not\in M^2\;\forall\;\tau\in G$ and any
maximal ideal $M$ of $S$, and thus $f(\tau,\tau^{-1})\not\in
M^2\;\forall\;\tau\in G$ and every maximal ideal $M$ of $S$, since
$f\!\sim_S\!\tilde{f}$.

If $A_f$ is not a maximal order, then it is the intersection of
finitely many maximal orders, say $A_{f_1},A_{f_2},\ldots, A_{f_l}$.
Note that
$$A_{f_i}=\sum_{\tau\in G}Sy^{(i)}_{\tau}=\sum_{\tau\in
G}Sk_{\tau}^{(i)}x_{\tau},$$ for some $k_{\tau}^{(i)}\in K$. Fix a
$\sigma\in G$, and a maximal ideal $N$ of $S$. Let $v_N$ be the
valuation corresponding to $N$, with value group $\mathbb{Z}$. Since
$$S=\bigcap_{i=1}^lSk_{\sigma}^{(i)},$$ there exists $i_0$
such that $v_N(k_{\sigma}^{(i_0)})=0$. Let $g=f_{i_0}$ and, for
$\tau \in G$, let $k_{\tau}=k_{\tau}^{(i_0)}$ and
$y_{\tau}=y_{\tau}^{(i_0)}$, so that $A_g=\sum_{\tau\in
G}Sk_{\tau}x_{\tau}=\sum_{\tau\in G}Sy_{\tau}$. By \cite[Proposition
3.1]{H}, $J(A_f)=\sum_{\tau\in G}I_{\tau}x_{\tau}$ and
$J(A_g)=\sum_{\tau\in G}J_{\tau}y_{\tau}$, where
$$I_{\tau}=\hspace*{-1.5em}\prod_{f(\tau,\tau^{-1})\not\in
M}\hspace*{-1.5em}M\;\;\;\;\;
\textrm{and}\;\;\;\;\;J_{\tau}=\hspace*{-1.5em}\prod_{g(\tau,\tau^{-1})\not\in
M}\hspace*{-1.5em}M,$$ and $M$ denotes a maximal ideal of $S$. Since
$A_f$ is a hereditary $V$-order in $\Sigma_f$ and $A_f\subseteq
A_g\subseteq\Sigma_f$, we have $J(A_g)\subseteq J(A_f)$, from which
we conclude that $J_{\sigma^{-1}}y_{\sigma^{-1}}\subseteq
I_{\sigma^{-1}}x_{\sigma^{-1}}$ and so
$J_{\sigma^{-1}}k_{\sigma^{-1}}\subseteq I_{\sigma^{-1}}.$ We have
$y_{\sigma^{-1}}J_{\sigma}y_{\sigma}=k_{\sigma^{-1}}x_{\sigma^{-1}}J_{\sigma}k_{\sigma}x_{\sigma}
=
J_{\sigma^{-1}}k_{\sigma^{-1}}x_{\sigma^{-1}}k_{\sigma}x_{\sigma}\subseteq
I_{\sigma^{-1}}x_{\sigma^{-1}}k_{\sigma}x_{\sigma}=\sigma^{-1}(k_{\sigma})I_{\sigma^{-1}}f(\sigma^{-1},\sigma)\\
=\left(k_{\sigma}I_{\sigma}f(\sigma,\sigma^{-1})\right)^{\sigma^{-1}}.$
On the other hand, $y_{\sigma^{-1}}J_{\sigma}y_{\sigma}=J^{\sigma^{-1}}_{\sigma}g(\sigma^{-1},\sigma)=\\
J_{\sigma^{-1}}g(\sigma^{-1},\sigma)$. Since $A_g$ is a maximal
order and therefore $g(\sigma^{-1},\sigma)\not\in M^2$ for every
maximal ideal $M$ of $S$, we see that
$J_{\sigma^{-1}}g(\sigma^{-1},\sigma)=J(V)S$ and so
$y_{\sigma^{-1}}J_{\sigma}y_{\sigma}=J(V)S$.
Therefore $J(V)S\subseteq
k_{\sigma}I_{\sigma}f(\sigma,\sigma^{-1})$. Since
$v_N(k_{\sigma})=0$, we conclude that $f(\sigma,\sigma^{-1})\not\in
N^2$, and so $f(\tau,\tau^{-1})\not\in M^2\;\forall\;\tau\in G$ and
any maximal ideal $M$ of $S$.

Conversely, suppose that $f(\tau,\tau^{-1})\not\in M^2$ for every
maximal ideal $M$ of $S$ and every $\tau\in G$. Let $B=O_l(J(A_f))$,
the left order of $J(A_f)$; that is, $B=\{x\in\Sigma_f\mid
xJ(A_f)\subseteq J(A_f)\}$. Since $\Sigma_f\supseteq B\supseteq
A_f$, $B=\sum_{\tau\in G}Sk_{\tau}x_{\tau}$, for some $k_{\tau}\in
K^{\#}$. For each $\tau\in G$, we have $S\subseteq Sk_{\tau}$, and
we will now show that $S=Sk_{\tau}$. As above, write $J(A_f)=\sum
I_{\tau}x_{\tau}$, with $I_{\tau}=\prod M$, where the product is
taken over all maximal ideals $M$ of $S$ for which
$f(\tau,\tau^{-1})\not\in M$. Observe that $J(V)S=I_1\supseteq
k_{\tau}x_{\tau}I_{\tau^{-1}}x_{\tau^{-1}}=k_{\tau}I^{\tau}_{\tau^{-1}}f(\tau,\tau^{-1})
=k_{\tau}I_{\tau}f(\tau,\tau^{-1}).$ Since $f(\tau,\tau^{-1})\not\in
M^2$ for every maximal ideal $M$ of $S$, we must have
$I_{\tau}f(\tau,\tau^{-1})=J(V)S,$ and so $J(V)S\supseteq
k_{\tau}J(V)S\supseteq J(V)S$ and thus $S=Sk_{\tau}$, as desired.
This shows that $O_l(J(A_f))=A_f$ and $A_f$ is
hereditary.\end{proof}

Not only can this criterion enable one to rapidly determine whether
or not the crossed-product order $A_f$ is hereditary, the utility of
the theorem above is now demonstrated by the ease with which the
following corollaries of it are obtained.

\begin{cor1} The crossed-product order $A_f$ is hereditary if and only if
$f(\tau,\gamma)\not\in M^2$ for all $\tau,\gamma\in G$ and every
maximal ideal $M$ of $S$. \end{cor1}

\begin{proof} This follows from the cocycle identity
$f^{\tau}(\tau^{-1},\tau\gamma)f(\tau,\gamma)=f(\tau,\tau^{-1})$.\end{proof}

In other words, the order $A_f$ is hereditary if and only if the values
of the two-cocycle $f$ are all square-free.

Since $A_f$ is a maximal order if and only if it is hereditary and
primary, by combining our result and results in \cite{H}, we
immediately have the following.

\begin{cor2} Given a crossed-product order $A_f$,
\begin{enumerate}\item it is a maximal order if and only if for every maximal ideal $M$ of $S$,
$f(\tau,\tau^{-1})\not\in M^2$ for all $\tau\in G$, and there exists
a set of right coset representatives $g_1,g_2,\ldots,g_r$ of $D_M$
in $G$ (i.e., $G$ is the disjoint union $\cup_iD_Mg_i$) such that
for all $i$, $f(g_i,g_i^{-1})\not\in M$.
\item if $S$ is a DVR, then it is a maximal order if and only if
$f(\tau,\tau^{-1})\not\in J(S)^2$ for all $\tau\in G$.
\end{enumerate}
\end{cor2}

\begin{proof} In either case, the primarity of $A_f$ is guaranteed
by \cite[Theorem 3.2]{H} (see also \cite[Proposition 2.1(b)]{H} when
$S$ is a DVR).\end{proof}

The Theorem above can readily be put to effective use with the
crossed-product orders in \cite[\S 4]{H}, for example. In that
section, all the crossed-product orders involved are primary orders,
and the two-cocycles are given in tabular form, with the values
factorized into primes of $S$. Using our criterion, it now becomes
a straightforward process to determine which of those orders are
maximal orders and which are not, by simply consulting, in each
case, the given table of values for the two-cocycle; the table whose
entries are all square-free represents a maximal order. This
determination can be made with little effort! In fact, if one knows
that the crossed-product order $A_f$ is a primary order, then
determining whether or not it is a maximal order could even be
easier, as the following result shows.

\begin{cor3} Suppose the crossed-product order $A_f$ is primary.
Then it is a maximal order if and only if there exists a maximal
ideal $M$ of $S$ such that $f(\tau,\tau^{-1})\not\in
M^2$ for all $\tau\in D_M$. \end{cor3}

\begin{proof} This follows from \cite[Corollary 3.11 and
Proposition 2.1(b)]{H}.\end{proof}

Let $L$ be an intermediate field of $F$ and $K$, let $G_L$ be the
Galois group of $K$ over $L$, let $U$ be a valuation ring of $L$
lying over $V$, and let $T$ be the integral closure of $U$ in $K$.
Then one can obtain a two-cocycle $f_{L,U}:G_L\times G_L\mapsto
T^{\#}$ from $f$ by restricting $f$ to $G_L\times G_L$ and
embedding $S^{\#}$ in $T^{\#}$. As before,
$A_{f_{L,U}}=\sum_{\tau\in G_L}Tx_{\tau}$ is a $U$-order in
$\Sigma_{f_{L,U}}=\sum_{\tau\in G_L}Kx_{\tau}=(K/L,G_L,f_{L,U})$.

\begin{cor4} Suppose the crossed-product order $A_f$ is
hereditary. Then $A_{f_{L,U}}$ is a hereditary order in
$\Sigma_{f_{L,U}}$ for each intermediate field $L$ of $F$ and $K$
and for every valuation ring $U$ of $L$ lying over $V$.
\end{cor4}

This leads to the following.

\begin{cor5} Suppose the crossed-product order $A_f$ is hereditary.
Then $A_{f_M}$ is a maximal order in $\Sigma_{f_M}$ for each maximal
ideal $M$ of $S$.
\end{cor5}

\begin{proof} The order $A_{f_M}$ is always primary, by \cite[Proposition 2.1(b)]{H}.\end{proof}

The following example illustrates two limitations of our theory,
however.

\begin{eg}
We give two crossed-product orders $A_{f_1}$ and $A_{f_2}$ with
$f_1\!\sim_K\!f_2$ and the graphs of $f_1$ and $f_2$ identical, but
$A_{f_1}$ is hereditary while $A_{f_2}$ is not. Also, we give an
example to demonstrate that the converse of Corollary 5 does not always hold.
\end{eg}

Let $F=\mathbb{Q}(x)$, and let $K=\mathbb{Q}(i)(x)$. Then the Galois
group $G=<\sigma>$ is a group of order two, where $\sigma$ is
induced by the complex conjugation on $\mathbb{Q}(i)$. If
$V=\mathbb{Q}[x]_{(x^2+1)}$, then $S$ has two maximal ideals, namely
$M_1=(x+i)S$ and $M_2=(x-i)S$, and $D_{M_1}=D_{M_2}=\{1\}$. Let
$f_1,f_2:G\times G\mapsto S^{\#}$ be two-cocycles defined by
$f_j(1,1)=f_j(1,\sigma)=f_j(\sigma,1)=1$ and
$f_1(\sigma,\sigma)=(x^2+1)x$, $f_2(\sigma,\sigma)=(x^2+1)^2x$.

Then $f_1\!\sim_K\!f_2$, and the subgroup of $G$ associated with
either cocycle is $H=\{1\}$, so that the graphs of $f_1$ and $f_2$
are identical. Clearly, $A_{f_1}$ is hereditary but $A_{f_2}$ is
not. We conclude that the property that a crossed-product order
$A_f$ is hereditary is not an intrinsic property of the graph of
$f$.

Also, if we set $f=f_2$, we see that $A_{f_M}=S_M$ for each maximal
ideal $M$ of $S$, and therefore $A_{f_M}$ is a maximal order in
$\Sigma_{f_M}=K$ for each maximal ideal $M$ of $S$, and yet $A_f$ is
not even hereditary (cf. \cite[Corollary 3.11]{H}, and \cite[Theorem
1]{Hd}). This is the case because $A_f$ is not primary, and also
because $f(G\times G)\not\subseteq U(S)$. $\Box$

 \bibliographystyle{amsplain}

\noindent
Department of Mathematics\\ Faculty of Science\\
Universiti Brunei Darussalam\\ Bandar Seri
Begawan BE1410\\ BRUNEI.
\\
\textit{Email address:} john.kauta@ubd.edu.bn

\end{document}